\documentclass[a4paper,10pt]{article}
\usepackage[utf8]{inputenc}
\usepackage[english]{babel}

\usepackage[nottoc]{tocbibind}

\usepackage{amsmath}
\usepackage{booktabs}
\usepackage{mathtools}
\usepackage{amssymb,epsfig,multirow}
\usepackage{amsbsy}
\usepackage{appendix}
\usepackage{graphicx}
\usepackage{authblk}
\usepackage{geometry}
\usepackage{url}
\usepackage{float}
\usepackage{color}
\usepackage{enumerate}
\usepackage{algorithmic}
\usepackage[linesnumbered,lined,boxed,commentsnumbered]{algorithm2e}

\newtheorem{thm}{Theorem}[section]

\newtheorem{prop}{Proposition}[section]

\newtheorem{exam}{Example}[section]
\newtheorem{assum}{Assumption}[section]

\newtheorem{rem}{Remark}[section]

\newenvironment{proof}{\noindent {\bf Proof. }}{\hfill $\Box$ \newline\par}

\title{KKT-based primal-dual exactness conditions for the Shor relaxation}

\author{M.Locatelli$^{1}$ \\
       { \small $^{1}$Dipartimento di Ingegneria e Architettura, Universit\`a di Parma}\\{ \small Parco Area delle Scienze, 181/A, Parma, Italy} \\
}

\date{}
\begin{document}
\maketitle
\begin{abstract}
In this work we present some exactness conditions for the Shor relaxation of diagonal QCQPs, which extend the conditions introduced in different recent papers about the same topic.
It is shown that the Shor relaxation is equivalent to two convex quadratic relaxations. Then, sufficient conditions for the exactness of the relaxations are derived from their KKT systems. It will be shown that, in some cases, by this derivation previous conditions in the literature, which can be viewed as {\em dual} conditions, since they only involve the Lagrange multipliers appearing in the KKT systems, 
%or {\em primal} conditions, since they only involve the original variables of the problems, 
can be extended to {\em primal-dual} conditions, which also involve the primal variables appearing in the KKT systems.
\end{abstract}
\section{Introduction}
In the recent literature different results about the exactness of the Shor relaxation (see \cite{Shor87}) for Quadratically Constrained Quadratic Programming (QCQP in what follows) problems have been proposed. The Shor relaxation can be proved to be exact for the Generalized Trust Region Subproblem (GTRS), where a single (not necessarily convex) quadratic inequality constraint is present. The exactness proof can be derived from a result discussed in \cite{Fradkov79}.
For other QCQPs the Shor relaxation is not always exact and different papers introduce conditions under which exactness holds for sub-classes of QCQPs. Some exactness results for the case of QCQPs with two quadratic constraints have been presented in \cite{Ye03}, while in \cite{Ai09} a necessary and sufficient condition for the exactness of the related Lagrangian dual has been given. Note that the case with two quadratic constraints, which includes the well known Celis-Dennis-Tapia (CDT) problem, has been recently proved to be polynomially solvable in different works \cite{Bienstock16,Consolini17,Sakaue16}. However, both the polynomial approaches proposed in \cite{Consolini17,Sakaue16}, based on the enumeration of all KKT points via the solution of bivariate polynomial systems, and the polynomial approach proposed in \cite{Bienstock16}, based on Barvinok's construction, have a limited practical applicability due to the large exponent of the polynomials appearing in the complexity result.
For QCQPs with a single unit ball constraint and further linear constraints, in \cite{Jeyakumar14} a dimension condition establishing exactness of the Shor relaxation is introduced. In \cite{BenTal14} a Second Order Cone Programming (SOCP) relaxation for the same problem has been discussed, while in \cite{Locatelli16} it has been shown that such relaxation is equivalent to the Shor relaxation. By the analysis of the KKT conditions for the SOCP relaxation, in \cite{Locatelli16} a condition more general than the dimension condition presented in \cite{Jeyakumar14} has been given. Note that in \cite{Burer13,Sturm03} an exact convex relaxation, obtained by adding to the Shor relaxation a so called SOC-RLT constraint, has been introduced, while in \cite{Burer15} the result has been extended to a generic number of linear constraints provided that these constraints have an empty intersection inside the unit ball. It is also worthwhile to mention that a polynomial-time algorithm for the solution of this problem (possibly also with the addition of further ball and reverse ball constraints) has been proposed under the assumption that the overall number of constraints is fixed. The approach is based on an enumeration of all possible KKT points. 
\newline\newline\noindent 
In this paper we are interested in deriving exactness conditions of the Shor relaxation in case of diagonal QCQPs, i.e., quadratic problems where the Hessian of all quadratic functions is diagonal or can be made diagonal after a change of variables (the Hessian matrices are simultaneously diagonalizable). In what follows we assume that the QCQP problem is already given in diagonal form.
Throughout the paper $N=\{1,\ldots,n\}$ will be the index set of the variables, and $M=\{1,\ldots,m\}$ will be the index set of the constraints. For a given symmetric matrix ${\bf Y}$, the notation ${\bf Y}\succeq {\bf O}$ means that the matrix is positive semidefinite. By $diag({\bf Y})$ we will denote the vector whose entries are the diagonal entries of matrix ${\bf Y}$. 
Then, 
a diagonal QCQP problem is the following:
\begin{equation}
\label{eq:qpproblem}
\begin{array}{lll}
c^\star=\min & {\bf x}^\top{\bf D} {\bf x} + 2 {\bf c}^\top{\bf x} & \\ [6pt]
& {\bf x}^\top {\bf A}^i {\bf x} + 2 {\bf a}_i^\top {\bf x} \leq b_i & i\in M,
\end{array}
\end{equation}
where matrix ${\bf D}$ and all matrices ${\bf A}_i$, $i\in M$, are diagonal. 
The classical 
Shor relaxation for this problem is:
\begin{equation}
\label{eq:Shor}
\begin{array}{lll}
v^\star=\min & {\bf D}\bullet {\bf X} + 2 {\bf c}^\top{\bf x} & \\ [6pt]
& {\bf A}^i \bullet {\bf X} + 2 {\bf a}_i^\top {\bf x} \leq b_i & i\in M \\ [6pt]
& {\bf X}-{\bf x}{\bf x}^\top \succeq {\bf O}.
\end{array}
\end{equation}
In \cite{BurerYe19} suitable assumptions are first introduced, namely:
\begin{assum}
\label{ass:1}
The following hold:
\begin{itemize}
\item The feasible region of (\ref{eq:qpproblem}) is nonempty;
\item $\exists\ \bar{{\bf y}}\geq {\bf 0}$ such that $\sum_{i\in M} \bar{y}_i {\bf A}_i \succ {\bf O}$;
\item The interior of the feasible region of (\ref{eq:Shor}) is nonempty.
\end{itemize}
\end{assum}
Then, some sufficient conditions are introduced under which there exists an optimal rank-one solution for the Shor relaxation, which is equivalent to prove that the Shor relaxation is exact, i.e., $v^\star=c^\star$. 
More precisely, in \cite{BurerYe19} it is proved that the
Shor relaxation is exact if
for each $k\in N$, the polyhedral set
%\begin{equation}
\begin{subequations}
\label{eq:setcond}
 %\begin{array}{ll}
\begin{align}
D_{kk} + \sum_{i\in M} \mu_i A^i_{kk}  =0 \label{eq:setcond1}\\ %[6pt]
c_k +\sum_{i\in M} \mu_i a_{ik} = 0 \label{eq:setcond2}\\ %[8pt]
D_{jj} + \sum_{i\in M} \mu_i A^i_{jj} \geq 0 &\ \ \  j\in N,\ j\neq k \label{eq:setcond3}\\ %[6pt]
\mu_i\geq 0 &\ \ \  i\in M, \label{eq:setcond4}
%\end{array}
\end{align}
\end{subequations}
%\end{equation}
is empty.
This result allows to re-derive a sign-definiteness condition presented in \cite{Sojoudi14},
stating that exactness holds if for all $j\in N$, it holds that $c_j$ and $a_{ij}$, $i\in M$, are all nonpositive or all nonnegative.
Moreover, 
for the relevant special case when ${\bf A}_i \in \{{\bf I},-{\bf I}, {\bf O}\}$ for each $i\in M$, i.e., when all constraints are ball, reverse ball, and linear constraints, in \cite{BurerYe19} it is shown that
exactness holds when the sign-definite condition is only satisfied by the variable corresponding to the lowest diagonal entry of matrix ${\bf D}$. Note that this special case is addressed also in \cite{Beck17}, where a branch-and-bound approach for its solution is proposed and an application to source localization problems is presented.
A further very recent result has been proved in \cite{Karzan20}, where a class of problems larger than the class of diagonal QCQPs is considered. We briefly discuss the condition introduced in that paper, only in the case of inequality constraints, although also equality constraints may be included. Note that in this case matrices ${\bf D}$ and ${\bf A}_i$, $i\in M$, are not necessarily diagonal. 
Let
$$
{\bf A}(\boldsymbol{\gamma})={\bf D}+\sum_{i\in M}\gamma_i {\bf A}_i,\ \ \ {\bf b}(\boldsymbol{\gamma})={\bf c}+\sum_{i\in M} \gamma_i {\bf a}_i.
$$
Let
$$
\Gamma=\{\boldsymbol{\gamma}:\  {\bf A}(\boldsymbol{\gamma})\succeq {\bf O},\ \boldsymbol{\gamma}\geq {\bf 0}\}.
$$
A face ${\cal F}$ of $\Gamma$ which does not contain any $\boldsymbol{\gamma}$ such that ${\bf A}(\boldsymbol{\gamma})\succ {\bf O}$ is called a semidefinite face, and the zero eigenspace of ${\cal F}$ is
$$
{\cal V}({\cal F})=\{{\bf x}\ :\ {\bf A}(\boldsymbol{\gamma}){\bf x}={\bf 0},\ \forall \boldsymbol{\gamma}\in {\cal F}\}.
$$
In \cite{Karzan20} it is assumed that $\Gamma$ is a polyhedral set. While this assumption is always fulfilled for diagonal QCQPs, it is also shown that it may hold also for non-diagonal QCQPs, but it is pointed out
that it is CoNP-hard to decide whether the assumption holds in general. Exactness of the Shor relaxation is proved under the condition that 
there exists some infinite sequence $\{{\bf h}^k\}$ such that ${\bf h}^k\rightarrow {\bf 0}$ (see the perturbation argument below) and for any $k$ and any semidefinite face ${\cal F}$ it holds that:
$$
{\bf 0}\not\in \{Proj_{{\cal V}({\cal F})} (b(\boldsymbol{\gamma})+{\bf h}^k),\ \ \boldsymbol{\gamma}\geq {\bf 0}\}.
$$
Note that in the same paper also some conditions are discussed under which 
the convex hull of the epigraph of the QCQP is given by the projection of the epigraph of its Shor relaxation.
Another recent result about this topic can be found in \cite{Jeyakumar18}. In that work minimax QCQPs are  considered, namely, the following problems are addressed
\begin{equation}
\label{eq:minimax}
\begin{array}{lll}
\min & \max_{r\in R} {\bf x}^\top{\bf D}^r {\bf x} + 2 {\bf c}_r^\top{\bf x} +c_{0r} & \\ [6pt]
& {\bf x}^\top {\bf A}^i {\bf x} + 2 {\bf a}_i^\top {\bf x} \leq b_i & i\in M,
\end{array}
\end{equation} 
where all matrices ${\bf D}^r$, $r\in R$, ${\bf A}^i$, $i\in M$, are diagonal ones, possibly obtained after the simultaneous diagonalization of all the Hessian matrices. Note that this class of problems is equivalent to the class of problems (\ref{eq:qpproblem}). Indeed, each problem (\ref{eq:qpproblem}) can be viewed as a special case of (\ref{eq:minimax}) by taking $|R|=1$, while, on the other hand, each problem (\ref{eq:minimax}) can be converted into an instance of problem
(\ref{eq:qpproblem}) after the addition of a variable $y$, which becomes the objective function to be minimized, and of the related constraints
$y\geq {\bf x}^\top{\bf D}^r {\bf x} + 2 {\bf c}_r^\top{\bf x} +c_{0r}$ for each $r\in R$.
In \cite{Jeyakumar18} a SOCP relaxation of problem (\ref{eq:minimax}) is introduced which is equivalent to the Lagrangian dual of this problem and, thus, also to the Shor relaxation (recall that the Lagrangian dual and the Shor relaxation are dual to each other and, thus, have the same optimal value if a constraint qualification holds).
In \cite{Jeyakumar18} an exactness condition is introduced based on the so called epigraphical set, defined as follows:
\begin{equation}
\label{eq:epigraph}
\begin{array}{ll}
E=\left\{({\bf w}, {\bf v})\in \mathbb{R}^{|R|+|M|}\ :\ \exists {\bf x}\in \mathbb{R}^{|N|}\ :\ \right. &
{\bf x}^\top{\bf D}^r {\bf x} + 2 {\bf c}_r^\top{\bf x} +c_{0r}\leq w_r,\ r\in R,  \\ [6pt]
& \left. {\bf x}^\top {\bf A}^i {\bf x} + 2 {\bf a}_i^\top {\bf x}\leq v_i,\ i\in M\right\}.
\end{array}
\end{equation}
%\begin{equation}
%\label{eq:epigraph}
%E=\{({\bf w}, {\bf v})\in \mathbb{R}^{|R|+|M|}\ :\ \exists {\bf x}\in \mathbb{R}^{|N|}\ :\ 
%{\bf x}^\top{\bf D}^r {\bf x} + 2 {\bf c}_r^\top{\bf x} +c_{0r}\leq w_r,\ r\in R,\ {\bf x}^\top {\bf A}^i {\bf x} + 2 {\bf a}_i^\top {\bf x}\leq v_i,\ i\in M\}.
%\end{equation}
It is shown that the SOCP relaxation is exact if the epigraphical set is closed and convex.
\newline\newline\noindent
In this paper we first show in Section \ref{sec:equivconv}, by a straightforward extension of the result proved in \cite{Locatelli16}, that for diagonal QCQPs the Shor relaxation is equivalent to a quadratic convex relaxation of problem (\ref{eq:qpproblem}).
Next, in Section \ref{sec:suff} the exactness condition related to the emptiness of the sets (\ref{eq:setcond}) is re-derived through an analysis of the KKT conditions of the convex relaxation. Moreover, it is shown how to strengthen the exactness condition in some cases and, in particular, in the already mentioned case when  ${\bf A}_i \in \{{\bf I},-{\bf I}, {\bf O}\}$ for each $i\in M$. It is shown through an example that
the new condition can be stronger than those discussed in \cite{BurerYe19}, \cite{Jeyakumar18} and \cite{Karzan20}.
Finally, in Section \ref{sec:further} a further equivalent convex relaxation is introduced and it is shown that KKT conditions for this relaxation allow to define an exactness condition which can be more efficiently checked.
\section{A convex relaxation equivalent to the Shor relaxation}  
\label{sec:equivconv}
We first observe that, in view of the fact that all matrices are diagonal, we can reformulate problem (\ref{eq:qpproblem}) as follows:
$$
\begin{array}{lll}
\min & \sum_{j\in N} D_ {jj} z_j+ 2 \sum_{j\in N} c_j x_j & \\ [6pt]
& \sum_{j\in N} A^i_ {jj} z_j  + 2 \sum_{j\in N} a_{ij} x_j \leq b_i & i\in M \\[6pt]
& x_j^2 = z_j & j\in N.
\end{array}
$$
Thus, following \cite{BenTal14}, a convex relaxation can be obtained by replacing the last equality constraints with convex inequality constraints:
\begin{equation}
\label{eq:convrel}
\begin{array}{lll}
p^\star=\min & \sum_{j\in N} D_ {jj} z_j+ 2 \sum_{j\in N} c_j x_j & \\ [6pt]
& \sum_{j\in N} A^i_ {jj} z_j  + 2 \sum_{j\in N} a_{ij} x_j \leq b_i & i\in M \\[6pt]
& x_j^2 \leq z_j & j\in N.
\end{array}
\end{equation}
In \cite{Locatelli16} it was proved the equivalence between this relaxation and the Shor relaxation when ${\bf A}_1={\bf I}$, ${\bf a}_1={\bf 0}$, ${\bf A}_i={\bf O}$ for all $i\in M\setminus \{1\}$. The result can be extended in a quite straightforward way to the general problem (\ref{eq:qpproblem}) (see also the proof in \cite{Karzan20} and note that the result can also be obtained as a special case of some results on sparse semidefinite programming problems presented in \cite{Fukuda00}).
\begin{thm}
\label{thm:1}
It holds that $p^\star=v^\star$, i.e., the optimal values of the Shor relaxation (\ref{eq:Shor}) and of the convex relaxation (\ref{eq:convrel}) are equivalent.
\end{thm}
\begin{proof}
Let $({\bf x}^\star, {\bf z}^\star)$ be an optimal solution of (\ref{eq:convrel}).
Let ${\bf C}$ be a diagonal matrix such that 
$$
C_{jj}=z_j^\star-(x_j^\star)^2,\ \ \ j\in N.
$$
Then
$$
\left({\bf x}^\star {{\bf x}^\star}^\top+{\bf C}, {\bf x}^\star\right),
$$
is feasible for (\ref{eq:Shor}) and its function value is equal to $p^\star$, so that $v^\star\leq p^\star$ holds. Conversely, let 
$({\bf X}^\star, {\bf x}^\star)$ be an optimal solution of (\ref{eq:Shor}). Let $\bar{{\bf z}}=diag({\bf X}^\star)$.
Then, $(\bar{{\bf z}}, {\bf x}^\star)$ is a feasible solution of (\ref{eq:convrel}) with function value equal to $v^\star$, so that
$p^\star\leq v^\star$ holds. Then, we can conclude that $p^\star=v^\star$.
\end{proof}
Now, this equivalence result can be employed in order to establish exactness conditions for the Shor relaxation by the analysis of the KKT conditions of the convex relaxation. This will be the topic of the next section.
\newline\newline\noindent
Before proceeding we briefly introduce the perturbation argument already adopted in \cite{BurerYe19,Locatelli16,Karzan20}.  We will make extensive use of this argument in the following sections.
\begin{prop}
\label{prop:perturb}
Exactness of the Shor relaxation holds for a problem with data $({\bf D}, {\bf A}_i, {\bf a}_i,{\bf c}, {\bf b})$ exactness of the Shor if
it holds for an infinite sequence of problems with perturbed data   $({\bf D}+\Delta {\bf D}^k, {\bf A}_i+\Delta {\bf A}_i^k,{\bf a}_i+\Delta {\bf a}_i^k, {\bf c}+\Delta {\bf c}^k, {\bf b}+\Delta {\bf b}^k)$ such that $||\Delta {\bf D}^k||, ||\Delta {\bf A}_i^k||,||\Delta {\bf a}_i^k||,||\Delta {\bf c}^k||, ||\Delta {\bf b}^k||\rightarrow 0$.
\end{prop}
\begin{proof}
The result simply follows from the fact  that, by continuity, the optimal values of problem (\ref{eq:qpproblem}) with the perturbed data converge to the optimal value of the unperturbed problem, and the same holds for the optimal values of the corresponding Shor relaxations.
\end{proof}
\section{Sufficient conditions for exactness of the Shor relaxation}
\label{sec:suff}
As proved in Theorem \ref{thm:1}, proving exactness of the Shor relaxation is equivalent to prove exactness 
of the convex relaxation (\ref{eq:convrel}). Under Assumption \ref{ass:1}, which will be assumed throughout the paper, optimal solutions of the convex problem (\ref{eq:convrel}) fulfill the corresponding  KKT conditions. In particular, we notice that
existence of an interior feasible solution $(\bar{{\bf X}}, \bar{{\bf x}})$ for problem (\ref{eq:Shor}) implies that also
the convex relaxation (\ref{eq:convrel}) admits an interior feasible point. Indeed, it is enough to consider the point 
$(diag(\bar{{\bf X}}), \bar{{\bf x}})$. %where $diag(\bar{{\bf X}})$ is the vector whose entries are the diagonal entries of
%matrix $\bar{{\bf X}}$. 
Then, Slater's condition holds and we can search the minimizer of problem (\ref{eq:convrel}) among the KKT points of the same problem. The KKT conditions are the following:
\begin{subequations}
\label{eq:KKT}
\begin{align} 
%\begin{equation}
%\begin{array}{ll}
D_{jj} + \sum_{i\in M} \mu_i A^i_{jj} -\nu_j =0 &\ \ \   j\in N \label{eq:KKT1} \\ %[6pt]
c_j +\sum_{i\in M} \mu_i a_{ij} +\nu_j x_j = 0 &\ \ \   j\in N \label{eq:KKT2} \\ %[6pt]
 \sum_{j\in N} A^i_ {jj} z_j  + 2 \sum_{j\in N} a_{ij} x_j \leq b_i &\ \ \   i\in M  \label{eq:KKT3}\\%[6pt]
x_j^2 \leq z_j & \ \ \  j\in N  \label{eq:KKT4}\\% [6pt]
\mu_i\left(b_i-\sum_{j\in N} A^i_ {jj} z_j  - 2 \sum_{j\in N} a_{ij} x_j \right)=0 &\ \ \   i\in M \label{eq:KKT5}\\% [6pt]
\nu_j (z_j-x_j^2)=0 &\ \ \   j\in N  \label{eq:KKT6}\\ %[6pt]
\mu_i, \nu_j\geq 0 &\ \ \   i\in M,\  j\in N.
%\end{array}
%\end{equation}
\end{align}
\end{subequations}
It obviously holds that the relaxation is exact if all constraints $x_j^2\leq z_j$, $j\in N$, are active at the optimal solution 
of (\ref{eq:convrel}). In view of the complementarity conditions (\ref{eq:KKT6}), this certainly holds if
$\nu_j>0$ for all $j\in N$. 
Note that $\nu_k=0$ for some $k\in N$ implies:
$$
\begin{array}{l}
D_{kk} + \sum_{i\in M} \mu_i A^i_{kk}  =0 \\ [6pt]
c_k +\sum_{i\in M} \mu_i a_{ik} = 0.
\end{array}
$$
Then, we can conclude that $\nu_j>0$ for all $j\in N$ if the condition stated in the following theorem holds. 
\begin{thm}
The convex relaxation (\ref{eq:convrel}) and the Shor relaxation are exact if the set (\ref{eq:setcond}) is empty for each $k\in N$.
%\begin{equation}
%\label{eq:setcond}
% \begin{array}{ll}
%D_{kk} + \sum_{i\in M} \mu_i A^i_{kk}  =0 \\ [6pt]
%c_k +\sum_{i\in M} \mu_i a_{ik} = 0 \\[8pt]
%D_{jj} + \sum_{i\in M} \mu_i A^i_{jj} \geq 0 & j\in N,\ j\neq k \\ [6pt]
%\mu_i\geq 0 & i\in M,
%\end{array}
%\end{equation}
%is empty.
\end{thm}
\begin{proof}
It is enough to observe that under this condition no KKT point with $\nu_k=0$ for some $k\in N$ exists.
\end{proof}
Note that this exactness condition is exactly the one stated in Theorem 1 of \cite{BurerYe19}. But while in that work exactness was established by showing that under the given condition there exists an optimal rank-one solution for the Shor relaxation, here exactness is derived in a different way, through the KKT conditions of the equivalent convex relaxation. As we will see, this different derivation
will allow to derive more general exactness conditions.
%\section{More general conditions}
%\label{sec:special}
%CONTROLLARE TITOLO!
%As already mentioned, in \cite{BurerYe19}  the special case when ${\bf A}_i \in \{{\bf I},-{\bf I}, {\bf O}\}$ for each $i\in M$, is also addressed. In fact, this case can be viewed as a subclass of the following class of problems.
%Based on the sufficient condition guarenteeing exactness of the Shor relaxation, in \cite{BurerYe19} a result proved in \cite{Sojoudi14} is re-derived and a new one is introduced. The former result states exactness under the assumption that all linear terms are sign-definite, i.e., for all $j\in N$, it holds that $c_j$ and $a_{ij}$, $i\in M$, are all nonpositive or all nonnegative.
%The latter result is about the special case when ${\bf A}_i \in \{{\bf I},-{\bf I}, {\bf O}\}$ for each $i\in M$.
%In this case it is proved that the sign-definite condition is only required for the variable corresponding to the lowest diagonal entry of matrix ${\bf D}$. Both these results can be derived as special cases of the following result.
\newline\newline\noindent
Now, let us consider a partition $N_h$, $h\in H$, of set $N$ with the following property
\begin{equation}
\label{eq:partition}
\forall\ j,k\in N_h\ :\ \forall\ h\in H\ :\ A_{jj}^i=A_{kk}^i=\xi^{ih} \ \forall\ i\in M.
\end{equation}
Then, each set $N_h$ contains indexes of variables whose quadratic terms are all equal throughout the constraints (but not necessarily in the objective function).
For each $h\in H$, let 
\begin{equation}
\label{eq:minargmin}
j_h=\arg\min_{j\in N_h} D_{jj},\ \ \ d_h^*=\min_{j\in N_h} D_{jj},
\end{equation}
 i.e., $j_h$ is the index of the lowest diagonal entry of matrix ${\bf D}$ among those in $N_h$ and $d_h^*$ is the value of such entry. The minimum entry is assumed to be unique. Extensions to cases where
the minimum entry is not unique can be done by the perturbation argument stated in Proposition \ref{prop:perturb}. Let $N_H=\{j_h\ : \ h\in H\}\subseteq N$.
%a perturbation argument already adopted in \cite{BurerYe19,Locatelli16,Karzan20}. The perturbation argument states that for a problem with data $({\bf D}, {\bf A}_i, {\bf a}_i,{\bf c}, {\bf b})$ exactness of the Shor relaxation holds if
%it holds for an infinite sequence of problems with perturbed data   $({\bf D}+\Delta {\bf D}^k, {\bf A}_i+\Delta {\bf A}_i^k,{\bf a}_i+\Delta {\bf a}_i^k, {\bf c}+\Delta {\bf c}^k, {\bf b}+\Delta {\bf b}^k)$ such that $||\Delta {\bf D}^k||, ||\Delta {\bf A}_i^k||,||\Delta {\bf a}_i^k||,||\Delta {\bf c}^k||, ||\Delta {\bf b}^k||\rightarrow 0$. This follows from the fact  that the optimal values of problem (\ref{eq:qpproblem}) with the perturbed data converge to the optimal value of the same problem with the original data, and the same holds for the optimal values of the Shor relaxations. \newline\newline\noindent
Note that in the special case, discussed in \cite{BurerYe19}, when ${\bf A}_i \in \{{\bf I},-{\bf I}, {\bf O}\}$, we have that $|H|=1$, while in the most general case we have $|H|=|N|$.
We prove the following theorem. 
\begin{thm}
\label{thm:2}
The convex relaxation (\ref{eq:convrel}) and, thus, also the Shor relaxation, are exact if the set (\ref{eq:setcond}) is empty for all $k\in N_H$.
%for each $\bar{h}\in H$ the following sets are empty:
%\begin{equation}
%\label{eq:kktrev22}
%\begin{array}{ll}
%D_{j_{\bar{h}} j_{\bar{h}}} + \sum_{i\in M} \mu_i \xi^{i\bar{h}} =0 &  \\ [6pt]
%c_{j_{\bar{h}}} +\sum_{i\in M} \mu_i a_{ij_{\bar{h}}}  = 0 &  \\ [6pt]
%D_{jj} + \sum_{i\in M} \mu_i \xi^{ih} \geq D_{jj}-d_h^* & j\in N_h,\ h\in H \\ [6pt]
%\mu_i\geq 0& i\in M.
%\end{array}
%\end{equation}
In particular, exactness holds if each variable $x_{j_h}$, $h\in H$, is sign-definite,  i.e.,
$c_{j_h}, a_{ij_h}$, $i\in M$, have the same sign for all $h\in H$. 
\end{thm}
\begin{proof}
Equations (\ref{eq:KKT1}) in this case can be written as follows:
$$
D_{jj} + \sum_{i\in M} \mu_i \xi^{ih} -\nu_j =0,\ \ \  j\in N_h,\ h\in H.
$$
Then, for each $h\in H$:
\begin{equation}
\label{eq:nuj}
\nu_j-\nu_{j_h}=D_{jj}-d_h^*\ \ \ \ \forall j\in N_h.
\end{equation}
In view of the definition of $j_h$, we have $\nu_j>0$ for all $j\in N_h\setminus\{j_h\}$. Thus, emptiness of the set (\ref{eq:setcond}) needs only to be imposed for $k\in N_h$ (also note that constraints (\ref{eq:setcond3}) need only to be imposed for
$j\in N_h$).
%The KKT conditions in this case can be rewritten as follows:
%\begin{equation}
%\label{eq:kktrev}
%\begin{array}{ll}
%D_{jj} + \sum_{i\in M} \mu_i \xi^{ih} -\nu_j =0 & j\in N_h,\ h\in H \\ [6pt]
%c_j +\sum_{i\in M} \mu_i a_{ij} +\nu_j x_j = 0 & j\in N \\ [6pt]
% \sum_{h\in H} \sum_{j\in N_h} \xi^{ih} z_j  + 2 \sum_{j\in N} a_{ij} x_j \leq b_i & i\in M \\[6pt]
%x_j^2 \leq z_j & j\in N \\ [6pt]
%\mu_i\left(b_i-\sum_{h\in H} \sum_{j\in N_h} \xi^{ih} z_j  - 2 \sum_{j\in N} a_{ij} x_j \right)=0 & i\in M \\ [6pt]
%\nu_j (z_j-x_j^2)=0 & j\in N \\[6pt]
%\mu_i, \nu_j\geq 0 & i\in M,\ j\in N.
%\end{array}
%\end{equation}
For what concerns the sign-definiteness condition, let us assume that for each $h\in H$
$c_{j_h}, a_{ij_h}$, $i\in M$, have the same sign (in fact, by the perturbation argument stated in Proposition \ref{prop:perturb} we can also admit null coefficients). Then, equations (\ref{eq:setcond2}) and the nonnegativity requirements (\ref{eq:setcond4}) cannot be fulfilled together. 
%any feasible solution of the convex relaxation with $z_{j_h}> x_{j_h}^2$ can be improved (or, at least, not worsened in case $c_{j_h}=0$) by a further feasible solution obtained by increasing $x_{j_h}$ until equality $z_{j_h}= x_{j_h}^2$  holds, when $c_{j_h}\leq 0$, or by decreasing $x_{j_h}$ until equality holds, when $c_{j_h}\geq 0$.
\end{proof}
Note that this result includes both the one proved in \cite{Sojoudi14} and re-derived in \cite{BurerYe19} for the general case, for which $|H|=|N|$, i.e., each set $N_h$ is a singleton, and sign-definiteness is required for all variables, and the case proved in \cite{BurerYe19} where all matrices are identity matrices, opposite of identity matrices and null matrices, for which $|H|=1$ and sign-definiteness is only required for a single variable.
\newline\newline\noindent
The conditions presented in \cite{BurerYe19} and \cite{Karzan20} can be viewed, in terms of the KKT conditions for the convex problem (\ref{eq:convrel}), as {\em dual} exactness conditions, since they only involve the Lagrange multipliers
associated to the constraints. 
%Instead, the conditions discussed in  \cite{Jeyakumar18} can be viewed as {\em primal} exactness conditions, since they only involve the original variables.
But the KKT system also involves the original, primal, variables. 
So the question is whether we can include both the original variables and the Lagrange multipliers in order to define {\em primal-dual} exactness conditions. This will result in conditions
which are usually hard to check, but in some special cases we will be able to verify them in an efficient way. We will first discuss the general case and then mention those cases for which the conditions can be efficiently checked.
Let us consider an index $j_h$, $h\in H$. Then, it follows from ({\ref{eq:KKT1})-(\ref{eq:KKT2}) and  from (\ref{eq:nuj}) that
\begin{equation}
\label{eq:xj}
x_j^h(\boldsymbol{\mu})=
\left\{
\begin{array}{ll}
-\frac{c_j+\sum_{i\in M} \mu_i a_{ij}}{D_{jj}-d_h^*}& \forall j\in N_h\setminus \{j_h\} \\ [6pt]
-\frac{c_j+\sum_{i\in M} \mu_i a_{ij}}{D_{jj}+\sum_{i\in M} \xi^{ih'}\mu_i} & \forall j\in N_{h'}\setminus \{j_{h'}\},\ h'\neq h.
\end{array}
\right.
\end{equation}
It also follows from  ({\ref{eq:KKT1})-(\ref{eq:KKT2}) and  from (\ref{eq:KKT6}) that for $h'\in H\setminus \{h\}$:
$$
\begin{array}{l}
(d_{h'}^* + \sum_{i\in M} \mu_i \xi^{ih'})x_{j_{h'}}=-(c_{j_{h'}}+\sum_{i\in M} \mu_i a_{ij_{h'}}) \\ [6pt]
(d_{h'}^* + \sum_{i\in M} \mu_i \xi^{ih'})z_{j_{h'}}=-(c_{j_{h'}}+\sum_{i\in M} \mu_i a_{ij_{h'}}) x_{j_{h'}}.
\end{array}
$$
%Instead, for each $j\in N_{h'}$, $h'\neq h$, we have in view of (\ref{eq:KKT1}), (\ref{eq:KKT2}) and (\ref{eq:nuj}) that
%$$
%x_j=-\frac{c_j+\sum_{i\in M} \mu_i a_{ij}}{D_{jj}+\sum_{i\in M} A_{j_h j_h}\mu_i}.
%$$
Thus, after imposing (\ref{eq:KKT3}) and (\ref{eq:KKT4}) we end up with the following primal-dual exactness condition. The Shor relaxation is exact if for each $j_h\in N_H$, the following set is empty:
\begin{subequations}
\label{eq:setcondbis}
 %\begin{array}{ll}
\begin{align}
d_h^* + \sum_{i\in M} \mu_i \xi^{ih}  =0 \label{eq:setcondbis1}\\ %[6pt]
c_{j_h} +\sum_{i\in M} \mu_i a_{i j_h} = 0 \label{eq:setcondbis2}\\ %[8pt]
d_{h'}^* + \sum_{i\in M} \mu_i \xi^{ih'} \geq 0 &\ \ \  h'\in H\setminus \{h\} \label{eq:setcondbis3}\\ %[6pt]
(d_{h'}^* + \sum_{i\in M} \mu_i \xi^{ih'})x_{j_{h'}}=-(c_{j_{h'}}+\sum_{i\in M} \mu_i a_{ij_{h'}}) & \ \ \  h'\in H\setminus \{h\} \label{eq:setcondbis7}\\
(d_{h'}^* + \sum_{i\in M} \mu_i \xi^{ih'})z_{j_{h'}}=-(c_{j_{h'}}+\sum_{i\in M} \mu_i a_{ij_{h'}}) x_{j_{h'}} & \ \ \  h'\in H\setminus \{h\} \label{eq:setcondbis8} \\
\sum_{h'\in H} \xi^{ih'} \left[z_{j_{h'}}  + \sum_{j\in N_{h'}\setminus \{j_{h'}\}}  x_j^h(\boldsymbol{\mu})^2\right]   +2 \sum_{h'\in H} a_{ij_h'} x_{j_h'}+2\sum_{j\in N\setminus N_H} a_{ij} x_j^h(\boldsymbol{\mu})\leq b_i &\ \ \   i\in M  \label{eq:setcondbis5}\\%[6pt]
%x_j^h(\boldsymbol{\mu})^2 \leq z_j & \ \ \  j\in N_H \setminus\{j_h\} \label{eq:setcondbis6} \\
x_{j_{h'}}^2\leq z_{j_{h'}}  &\ \ \  h'\in H\\
\mu_i\geq 0 &\ \ \  i\in M. \label{eq:setcondbis4}
%\end{array}
\end{align}
\end{subequations}
Note that for $j\not\in N_H$, $\nu_j>0$, so that we could replace $z_j$ with $x_j^h(\boldsymbol{\mu})^2$.
Taking into account the definition (\ref{eq:xj}) for $x_j^h(\boldsymbol{\mu})$, the above sets can be seen as solution sets of a system of polynomial equations and inequalities, where the degree of the polynomials is at most $2n$.
Unfortunately, establishing whether these systems admit no solution or, equivalently, that the Shor relaxation is exact is, in general, a hard task. However, in what follows we will discuss some cases for which the condition can be efficiently checked.
\subsection{The cases $|M|=1$ and $|M|=2$}
\label{sec:M=2}
We briefly discuss the case $|M|=1$. This is the already mentioned GTRS problem for which it is well known that the Shor relaxation is always exact. Exactness can be viewed as an immediate consequence of the fact that the two equations (\ref{eq:setcondbis1}) and (\ref{eq:setcondbis2}), possibly after
the application of the perturbation argument stated in Proposition \ref{prop:perturb}, cannot be
fulfilled at the same time.
\newline\newline\noindent
When $|M|=2$ exactness does not always hold but this can be checked by easy computations. First note that in this case the system of two equations (\ref{eq:setcondbis1}) and (\ref{eq:setcondbis2}), again possibly after a perturbation of the data, admits a unique solution $\bar{\boldsymbol{\mu}}=(\bar{\mu}_1,\bar{\mu}_2)$. If the solution has at least one negative component, then nonnegativity of the $\mu$ values is violated and the set (\ref{eq:setcondbis}) is empty. By the usual perturbation argument, also cases with at least one null $\mu$ value are covered. So, we are only left with cases where both $\mu$ values are positive. This allows us to convert all inequalities (\ref{eq:setcondbis5}) into equations by exploiting the complementarity conditions (\ref{eq:KKT5}).
By the perturbation argument we have that all constraints (\ref{eq:setcondbis3}) are not active, so that, by 
(\ref{eq:setcondbis7}) and (\ref{eq:setcondbis8}) we can set:
$$
x_{j_{h'}}=-\frac{c_{j_{h'}}+\sum_{i\in M} \bar{\mu}_i a_{ij_{h'}}}{d_{h'}^* + \sum_{i\in M} \bar{\mu}_i \xi^{ih'}},\ \ \ z_{j_{h'}}=x_{j_{h'}}^2.
$$   
This way, in the two equations (\ref{eq:setcondbis5}) we just have the two unknowns $z_{j_h}$ and $x_{j_h}$. After solving this system, we can conclude that the set (\ref{eq:setcondbis}) is empty if  $x_{j_h}^2> z_{j_h}$ holds.
For the sake of illustration we derive the exactness condition 
in the case of trust region problems with one additional linear constraint. As already mentioned, for this problem in \cite{Burer13,Sturm03} an exact SOC-RLT relaxation is proposed. 
The Shor relaxation is not always exact but its exactness can be checked by a very simple condition.
The problem can always be converted into 
an instance of diagonal QCQP:
$$
\begin{array}{ll}
\min & \sum_{j\in N} D_{jj} x_j^2 + 2\sum_{j\in N}c_j x_j \\ [6pt]
      & \sum_{j\in N} x_j^2 \leq 1 \\ [6pt]
   & 2\sum_{j\in N} a_j x_j\leq b.
\end{array}
$$ 
Note that we can take $|H|=1$ in this case.
Exactness certainly holds if $c_{j_1}a_{j_1}\geq 0$ (sign-definiteness condition). If $c_{j_1}a_{j_1}< 0$, 
we have $\bar{\mu}_1=-d_{1}^*$ and $\bar{\mu}_2=-\frac{c_{j_1}}{a_{j_1}}$.
Then,
$$
\bar{x}_j=x_j^1(\bar{\mu}_1,\bar{\mu}_2)=-\frac{c_j-a_j\frac{c_{j_1}}{a_{j_1}}}{D_{jj}-d_1^*} \ \ \ \forall j \in N\setminus \{j_1\},
$$
while 
$$
\bar{x}_{j_1}=x^1_{j_1}(\bar{\mu}_1,\bar{\mu}_2)=\frac{b-2 \sum_{j\in N\setminus \{j_1\}} a_{j} \bar{x}_j}{a_{j_1}}.
$$
Finally, exactness of the convex relaxation holds if
\begin{equation}
\label{eq:exactonelin}
\sum_{j\in N} \bar{x}_j^2 \geq 1.
\end{equation}
\begin{rem}
In \cite{BurerYe19} it is proved that for a class of random general
QCQPs the probability of having an exact semidefinite relaxation converges to 1 as $|N|\rightarrow \infty$  as long as the number
of constraints grows no faster than a fixed polynomial in the number of variables. 
For QCQPs with a single quadratic constraint and a single linear constraint this fact emerges quite clearly from the above exactness condition. Indeed, under very mild assumptions on the random generation of the data, for some $j\in N$ there is a strictly positive probability $\ell>0$
that $\bar{x}_j\not\in (-1,1)$, and this is enough to guarantee  that the exactness condition (\ref{eq:exactonelin}) holds. Therefore, under the assumption of independent generation of the data, the probability of fulfilling 
the exactness condition is at least $1-(1-\ell)^{|N|}$, which converges to 1 as $|N|\rightarrow \infty$.
\end{rem}
\subsection{The case $|M|=3$} 
With a little more effort, exactness conditions can also be given for $|M|=3$. We first notice that
we can consider only points for which none of the inequalities (\ref{eq:setcondbis3}) is active. Indeed, if one of them were active, then by (\ref{eq:setcondbis7}) we should also have $c_{j_{h'}}+\sum_{i\in M} \mu_i a_{ij_{h'}}=0$, i.e., the three $\mu$ variables should fulfill four equations which, possibly after applying the perturbation argument, is not possible. By (\ref{eq:setcondbis1})-(\ref{eq:setcondbis2}) we have that at least two $\mu$ variables must be positive. Thus, we can consider four distinct cases:
i) $\mu_1,\mu_2>0,\mu_3=0$; ii) $\mu_1,\mu_3>0,\mu_2=0$; iii) $\mu_2,\mu_3>0,\mu_1=0$; iv) $\mu_1,\mu_2,\mu_3>0$. If case i) holds, then we can: derive $\mu_1,\mu_2$ from  (\ref{eq:setcondbis1})-(\ref{eq:setcondbis2}); 
check whether the computed values (together with $\mu_3=0$) fulfill the constraints (\ref{eq:setcondbis3}) and (\ref{eq:setcondbis4}); if yes,
then derive $x_j$, $j\in N_h\setminus \{j_h\}$ from (\ref{eq:setcondbis7}) and $z_j$ from (\ref{eq:setcondbis8}); impose, in view of (\ref{eq:KKT5}), that equality holds for constraints (\ref{eq:setcondbis5}) for $i=1,2$; derive $x_{j_h}$ and $z_{j_h}$ from such two equations; finally, check whether $x_{j_h}^2>z_{j_h}$. In a completely similar way we can deal with cases ii) and iii). In case iv), again in view of (\ref{eq:KKT5}) we notice that all three constraints (\ref{eq:setcondbis5}) must be active. In this case we have a system of three equations with two unknowns $x_{j_h}$ and $z_{j_h}$, which can be fulfilled only if one of the three equations can be obtained as a linear combination of the other two equations. In particular, this imply that the right-hand side of one of the equations is a given linear combination of the right-hand sides of the other two equations.   This results into an univariate polynomial equation, taking into account that by (\ref{eq:setcondbis1})-(\ref{eq:setcondbis2}) we can derive two of the three $\mu$ variables as a linear function of the remaining one. The roots of the univariate polynomial equation can be efficiently computed.
In principle, we could proceed in the same way for larger $|M|$ values, but  the resulting procedure tends to become quite inefficient with the need of solving multivariate polynomial systems.
\subsection{The case $|H|=1$, $|M|$ arbitrary}
\label{sec:h=1}
We discuss the special case when $|M|$ is arbitrary but $|H|=1$, so that for each $i\in M$, $A_{jj}^i=\xi_i$ for all $j\in N$.
The case when ${\bf A}_i \in \{{\bf I},-{\bf I}, {\bf O}\}$ for each $i\in M$, discussed in \cite{BurerYe19}, corresponds to $\xi^i\in \{0,-1,1\}$, for each $i\in M$.
%In fact, when ${\bf A}_i \in \{{\bf I},-{\bf I}, {\bf O}\}$ we can even define a more general condition under which 
%exactness holds. This is stated in the following proposition.
%\begin{prop}
%\label{prop:1}
%Let ${\bf A}_i \in \{{\bf I},-{\bf I}, {\bf O}\}$ for all $i\in M$.
% $J=\arg\min_{j\in N} D_{jj}$, and $d^*=\min_{j\in N} D_{jj}$. 
Based on the previous discussion, in this case the set (\ref{eq:setcondbis}), whose emptiness guarantees exactness of the Shor relaxation, is equivalent to: 
%VEDERE SE DERIVA DIRETTAMENTE DA DISCUSSIONE PRECEDENTE 
%exactness of the convex relaxation (\ref{eq:convrel}) holds if the following set is empty:
\begin{equation}
\label{eq:setspec}
\begin{array}{ll}
d_1^* +\sum_{i\in M} \mu_i \xi_{i}  = 0 &  \\ [6pt]
c_{j_1} +\sum_{i\in M} \mu_i a_{ij_1}  = 0 &  \\ [6pt]
\xi^{i} z_{j_1}+\sum_{j\neq j_1} \xi^i \left(-\frac{c_j +\sum_{i\in M} \mu_i a_{ij}}{D_{jj}-d^*}\right)^2  + 2 a_{ij_1} x_{j_1}-2\sum_{j\neq j_1}a_{ij}\frac{c_j +\sum_{i\in M} \mu_i a_{ij}}{D_{jj}-d^*}  \leq b_i & i\in M \\[6pt]
x_{j_1}^2 \leq z_{j_1} & \\ [6pt]
\mu_i \geq 0 & i\in M.
\end{array}
\end{equation}
%where $\xi^i\in \{0,-1,1\}$, $i\in M$.
%\end{prop}
%\begin{proof}
%In view of the first set of equalities in (\ref{eq:kktrev}) we have that
%\begin{equation}
%\label{eq:nujinh}
%\nu_j=\nu_{j_1}+(D_{jj}-d_1^*)\ \ \ \forall\ j\in N.
%\end{equation}
%If $\nu_{j_1}>0$, then $\nu_j>0$ for all $j\in N$ and exactness of the convex relaxation (\ref{eq:convrel}) holds. Therefore, let us assume that $\nu_{j_1}=0$ and, consequently, $\nu_j=D_{jj}-d_1^*>0$ for all $j\neq j_1$.
%Then, from the second set of equalities in (\ref{eq:kktrev}) we  have that
%$$
%x_j=-\frac{c_j +\sum_{i\in M} \mu_i a_{ij}}{D_{jj}-d_1^*},\ \ \ \forall\ j\neq j_1.
%$$
%After substitution in the third set of inequalities in (\ref{eq:kktrev}), observing that $\nu_j>0$ implies $z_j=x_j^2$ for all $j\neq j_1$, we are left with the set defined in (\ref{eq:setspec}). If such set is empty, then $\nu_{j_1}>0$ must hold and, thus, also exactness holds.
%\end{proof}  
A drawback of the above condition is that the set (\ref{eq:setspec}), defined by linear and quadratic inequalities, is not convex if $\xi_i<0$ for at least one $i\in M$. In the next section, we will introduce a further condition, at least as strong as this one, but only involving convex sets, so that the condition can be checked in polynomial time. Before that, in what follows we present a simple example where exactness can be established by the new condition but not through the conditions introduced in \cite{BurerYe19}, \cite{Jeyakumar18} and \cite{Karzan20}.
\begin{exam}
\label{ex:1}
Let us consider the following problem parameterized with respect to the right-hand side of the second constraints:
\begin{equation}
\label{eq:param}
\begin{array}{ll}
\min & -x_1^2 -\frac{1}{2}x_2^2+x_2 \\ [6pt]
& x_1^2+x_2^2 +x_1-x_2 \leq 2 \\ [6pt]
& -x_1+x_2\leq \xi. %-1-\varepsilon,
\end{array}
\end{equation}
%where $\varepsilon>0$ is a sufficiently small value such that the feasible region is not empty. 
The feasible set has a nonempty interior for $\xi\in (1-\sqrt{5},+\infty)$.
Now, the set (\ref{eq:setspec}) in this case is:
$$
\begin{array}{l}
-1+\mu_1=0 \\ [6pt]
\mu_1-\mu_2=0 \\ [6pt]
 z_1+1 +x_1+1 \leq 2 \\ [6pt]
% -x_1-1\leq -1-\varepsilon \\ [6pt]
 -x_1-1\leq \xi \\ [6pt]
x_1^2\leq z_1,
\end{array}
$$
which can be seen to be empty for $\xi<-1$, so that exactness of the convex relaxation (\ref{eq:convrel}) is established in these cases, while it is not empty (consider, e.g., $x_1=z_1=0$, $\mu_1=\mu_2=1$) for $\xi\geq -1$.
But exactness cannot be established by the conditions proposed in  \cite{BurerYe19}, \cite{Jeyakumar18}  and \cite{Karzan20}. 
Indeed, for what concerns the condition proposed in \cite{BurerYe19}, we notice that for $k=1$ the set (\ref{eq:setcond}) is:
$$
\begin{array}{l}
-1+\mu_1=0 \\ [6pt]
\mu_1-\mu_2=0 \\ [6pt]
\mu_1,\mu_2\geq 0,
\end{array}
$$
which is not empty (note that the sign-definiteness condition stated in Theorem \ref{thm:2} does not hold).
For what concerns the condition introduced in \cite{Jeyakumar18}, in this case
the epigraphical set (\ref{eq:epigraph}) is
$$
E=\{(w_1,v_1,v_2) :\ \exists (x_1,x_2)\ :\ \ -x_1^2 -\frac{1}{2}x_2^2+x_2\leq w_1,\ x_1^2+x_2^2 +x_1-x_2 \leq v_1,\  -x_1+x_2\leq v_2\}.
$$
It can be seen that the points $\left(-\frac{5}{2}, 4, -2\right)$ and $\left(-\frac{5}{2},2,0\right)$ belong to $E$ (consider $x_1=1, x_2=-1$ and $x_1=x_2=-1$, respectively). But their midpoint $\left(-\frac{5}{2},3,-1\right)$ does not belong to $E$, so that $E$ is not convex.
For what concerns the condition introduced in \cite{Karzan20}, we notice that in this case
we have 
$$
{\bf A}(\gamma_1,\gamma_2)=\left(\begin{array}{cc}
-1 + \gamma_1 & 0 \\
0 & -\frac{1}{2}+\gamma_1 
\end{array}
\right),\ \ \  {\bf b}(\gamma_1,\gamma_2)=\left(\begin{array}{c}
\gamma_1 -\gamma_2 \\
1- \gamma_1+\gamma_2 
\end{array}
\right),
$$
We also have the following semidefinite face:
$$
{\cal F}=\{(\gamma_1,\gamma_2)\ :\ \gamma_1=1,\ \ \ \gamma_2\geq 0\},
$$
so that
$$
{\cal V}({\cal F})=\{(t,0)\ :\ t\in \mathbb{R}\}.
$$
Then, the condition  introduced in \cite{Karzan20} requires that  
for some sequence $\{h^k\}$, with $h^k\rightarrow 0$, we have that
$$
0\not\in \{1-\gamma_2 +h^k,\ \ \gamma_2\geq 0\},
$$
which, however, does not hold.
Note that the exactness conditions in  \cite{BurerYe19}, \cite{Jeyakumar18}  and \cite{Karzan20} do not depend on the right-hand sides of the constraints. Thus, in this example all three conditions are not fulfilled for all possible $\xi$ values.
\end{exam}
\section{A further convex relaxation}
\label{sec:further}
The convex relaxation (\ref{eq:convrel}) can be further simplified when the set of variables can be partitioned as indicated
in (\ref{eq:partition}), where each set $N_h$ collects variables whose quadratic terms are equal throughout all the constraints. Recalling the definitions of $j_h$ and $d_h^*$ given in (\ref{eq:minargmin}), the new convex relaxation
is the following:
\begin{equation}
\label{eq:newconvrel}
\begin{array}{lll}
\min & \sum_{h\in H} d_h^* w_h +\sum_{h\in H} \sum_{j\in N_h} (D_{jj}-d_h^*)x_j^2 +2\sum_{h\in H} \sum_{j\in N_h} c_j x_j & \\ [6pt]
& \sum_{h\in H} \xi_{ih} w_h +2\sum_{h\in H} \sum_{j\in N_h} a_{ij} x_j\leq b_i & i\in M \\ [6pt]
& \sum_{j\in N_h} x_j^2 \leq w_h & h\in H.  
\end{array}
\end{equation}
Note that for $|H|=|N|$ this is the same as the convex relaxation (\ref{eq:convrel}). But for $|H|<|N|$ this relaxation requires
the addition of a lower number of additional variables and of related convex quadratic constraints.
The KKT conditions for such relaxation are:
\begin{subequations}
\label{eq:kktrevnew}
\begin{align}
d_{h}^* + \sum_{i\in M} \mu_i \xi^{ih} -\gamma_h =0 &\ \ \  h\in H \\ 
 (D_{jj}-d_h^*)x_j+c_j +\sum_{i\in M} \mu_i a_{ij} +\gamma_h x_j = 0 &\ \ \  j\in N_h, \ h\in H \\ 
 \sum_{h\in H}  \xi^{ih} w_h  + 2 \sum_{h\in H} \sum_{j\in N_h}  a_{ij} x_j \leq b_i &\ \ \  i\in M \\
\sum_{j\in N_h} x_j^2 \leq w_h &\ \ \  h\in H \\ 
\mu_i\left(b_i-\sum_{h\in H}  \xi^{ih} w_h - 2 \sum_{h\in H} \sum_{j\in N_h} a_{ij} x_j \right)=0 &\ \ \  i\in M \label{eq:kktrevnew1}\\ 
\gamma_h (w_h-\sum_{j\in N_h} x_j^2 )=0 & \ \ \ h\in H \\
\mu_i, \gamma_h\geq 0 &\ \ \  i\in M,\ h\in H.
\end{align}
\end{subequations}
We prove the following proposition stating that the optimal value of the new convex relaxation (\ref{eq:newconvrel}) is equal to the optimal value of the original convex relaxation (\ref{eq:convrel}) (and, as a consequence, also of the Shor relaxation).
\begin{prop}
\label{prop:1bbb}
The optimal values of the convex relaxations (\ref{eq:convrel}) and (\ref{eq:newconvrel}) are equal.
\end{prop}
\begin{proof}
We recall the KKT conditions (\ref{eq:KKT}) for the convex relaxation (\ref{eq:convrel}). 
Let $({\bf x}^\star,{\bf z}^\star,\boldsymbol{\mu}^\star,\boldsymbol{\nu}^\star)$ be a KKT point for that problem.
Recalling (\ref{eq:nuj}), we have: %Note that, by the first set of equations, (\ref{eq:nujinh}) holds, i.e.,
$$
\nu_j^\star=\nu_{j_h}^\star+(D_{jj}-d_h^*),\ \ \ j\in N_h \setminus \{j_h\}.
$$
Now, let us set for each $h\in H$:
$$
\begin{array}{l}
w_h^\star=\sum_{j\in N_h} z_j^\star \\ [6pt]
\gamma_h^\star=\nu_{j_h}^\star.
\end{array}
$$
Then, $({\bf x}^\star,{\bf w}^\star,\boldsymbol{\mu}^\star,\boldsymbol{\gamma}^\star)$ is a KKT point for the new relaxation.
Indeed, for each $h\in H$
$$
d_{h}^*  +\sum_{i\in M} \xi_{ih} \mu^\star_i-\nu_{j_h}^\star=d_h^*  +\sum_{i\in M} \xi_{ih} \mu^\star_i-\gamma_h^\star=0.
$$
For each $j\in N$:
$$
\begin{array}{l}
c_j + \sum_{i\in M} a_{ij} \mu^\star_i +\nu^\star_j x_j^\star=c_j + \sum_{i\in M} a_{ij} \mu^\star_i +(\nu^\star_{j_h}+D_{jj}-d_h^*) x_j^\star= \\ [8pt]
=c_j + \sum_{i\in M} a_{ij} \mu^\star_i +\gamma_h^\star x_j^\star+(D_{jj}-d_h^*) x_j^\star=0.
\end{array}
$$
For each $i\in M$:
$$
\sum_{h\in H} \xi_{ih} \sum_{j\in N_h} z_j ^\star + 2 \sum_{j\in N} a_{ij} x_j^\star =\sum_{h\in H}\xi_{ih}  w_h^\star + 2 \sum_{j\in N} a_{ij} x_j^\star    \leq  b_i,
$$
and
$$
\mu_i^*\left( \sum_{h\in H} \xi_{ih}\sum_{j\in N_h} z_j^\star + 2 \sum_{j\in N} a_{ij} x_j^\star - b_i\right)=\mu_i^*\left( \sum_{h\in H}\xi_{ih}  w_h^\star + 2 \sum_{j\in N} a_{ij} x_j^\star - b_i\right)=0.
$$
Moreover,
$$
(x_j^\star)^2 \leq z_j^\star \ \   j\in N_h\ \ \Rightarrow\ \ \sum_{j\in N_h}  (x_j^\star)^2\leq \sum_{j\in N_h}   z_j^\star\ \ \Rightarrow\ \ \sum_{j\in N_h}  (x_j^\star)^2\leq w_h^\star,
$$
while, for $j\in N_h\setminus \{j_h\}$, $\nu_j^\star>0$ holds, so that $(x_j^\star)^2 = z_j^\star$. Then, 
$$
\nu_{j_h}^\star[(x_j^\star)^2 - z_j^\star]=0\ \ \ \forall j\in N_h,
$$
and summing all these up:
$$
\nu_{j_h}^\star\left[\sum_{j\in N_h} (x_j^\star)^2 - \sum_{j\in N_h} z_j^\star\right]=\gamma_h^\star\left[\sum_{j\in N_h} (x_j^\star)^2 - w_h^\star\right]=0.
$$
It also holds that the objective function value of the original relaxation in $({\bf x}^\star,{\bf z}^\star)$ is equal to
the objective function value of the new relaxation in $({\bf x}^\star,{\bf w}^\star)$, thus establishing the equivalence between the two relaxations.
\end{proof}
If we consider the special case $|H|=1$, which includes the case when ${\bf A}_i \in \{{\bf I},-{\bf I}, {\bf O}\}$, then only
a single additional variable $w_1$ needs to be introduced. Without loss of generality, we assume that $j_1=1$.
Then, we have the following result.
\begin{prop}
\label{prop:equiv}
For $|H|=1$ the convex relaxation (\ref{eq:newconvrel}) is exact if the following convex set is empty:
\begin{equation}
\label{eq:newcond}
\begin{array}{ll}
d_{1}^*  +\sum_{i\in M} \xi_i \mu_i=0 & \\ [6pt]
c_1 + \sum_{i\in M} a_{i1} \mu_i=0 &  \\[6pt]
(D_{jj}-d_{1}^*) x_j + c_j + \sum_{i\in M} a_{ij} \mu_i=0 & j\in N\setminus \{1\} \\[6pt]
\xi_i w_1  + 2 \sum_{j\in N} a_{ij} x_j \leq b_i & i\in M \\[6pt]
\sum_{j\in N} x_j^2 \leq w_1 &   \\ [6pt]
\mu_i\geq 0 & i\in M.
\end{array}
\end{equation}
\end{prop}
\begin{proof}
The relaxation (\ref{eq:newconvrel}) is not exact if $\nu_1=0$. Then, if the convex set (\ref{eq:newcond}) is empty, it follows from (\ref{eq:kktrevnew}) with $|H|=1$, $N_1=N$, that no KKT point with $\nu_1=0$ exists.
\end{proof}
Note that this condition can be checked more efficiently than the one stated in Section \ref{sec:h=1} (with $j_1=1$) (since (\ref{eq:newcond}) is a convex set) and is at least as strong as that condition. % one stated in Section \ref{sec:h=1} (with $j_1=1$).
Indeed, if $(\bar{\boldsymbol{\mu}},\bar{{\bf x}},\bar{w}_1)$ belongs to the set (\ref{eq:newcond}), then
$(\bar{\boldsymbol{\mu}},\bar{{\bf x}},\bar{{\bf z}})$, where
$$
\bar{z}_j=\bar{x}_j^2,\ \ \ j\neq 1,\ \ \ \bar{z}_1=\bar{w}_1-\sum_{j\neq 1} \bar{x}_j^2,
$$
belongs to the set (\ref{eq:setspec}).
\newline\newline\noindent
In fact, in (\ref{eq:newcond}) we could replace $\sum_{j\in N} x_j^2 \leq w$ with $\sum_{j\in N} x_j^2 < w$. Indeed, if the set defined in (\ref{eq:newcond}) is not empty but only contains points for which equality holds, then the relaxation is still exact.
Thus, we could reformulate Proposition \ref{prop:equiv} in this slightly stronger way.
\begin{prop}
\label{prop:equiv1}
For $|H|=1$ the convex relaxation (\ref{eq:newconvrel}) is exact if the following convex problem has a nonnegative optimal value.
\begin{equation}
\label{eq:newcond1}
\begin{array}{lll}
\min & \sum_{j\in N} x_j^2 - w_1 & \\ [6pt]
& d_{1}^*  +\sum_{i\in M} \xi_i \mu_i=0 & \\ [6pt]
& c_1 + \sum_{i\in M} a_{i1} \mu_i=0 &  \\[6pt]
& (D_{jj}-d_{1}^*) x_j + c_j + \sum_{i\in M} a_{ij} \mu_i=0 & j\in N\setminus \{1\} \\[6pt]
& \xi_i w_1  + 2 \sum_{j\in N} a_{ij} x_j \leq b_i & i\in M \\[6pt]
& \mu_i\geq 0 & i\in M.
\end{array}
\end{equation}
\end{prop}
Up to now we have basically ignored the complementarity conditions (\ref{eq:kktrevnew1}). We can strengthen 
condition (\ref{eq:newcond1}) by taking them into account. We first notice that for each $j\in N\setminus\{1\}$ we can set
$$
x_j(\boldsymbol{\mu})=-\frac{c_j + \sum_{i\in M} a_{ij} \mu_i}{D_{jj}-d_{1}^*}.
$$
Note that, possibly after the application of the perturbation argument, we must have that at least two $\mu$ values are strictly positive.
Indeed, both the equation  $d_{1}^*  +\sum_{i\in M} \xi_i \mu_i=0$ and the equation  $c_1 + \sum_{i\in M} a_{i1} \mu_i=0$ must be fulfilled and, possibly after an arbitrarily small perturbation of the data, such equations can not be fulfilled
if all but one of the $\mu$ values are equal to 0.
In view of (\ref{eq:kktrevnew1}) for positive $\mu$ values the corresponding inequality constraints are active. Thus, we can enumerate all possible
$I\subseteq M$ with cardinality two and ask that the following convex problems have nonnegative optimal value:
$$
\begin{array}{lll}
\min & \sum_{j\in N\setminus\{1\}}  x_j(\boldsymbol{\mu})^2 +x_1^2 - w_1 & \\ [6pt]
&d_{1}^*  +\sum_{i\in M} \xi_i \mu_i=0 & \\ [6pt]
&c_1 + \sum_{i\in M} a_{i1} \mu_i=0 &  \\[6pt]
&\xi_i w_1  + 2 a_{i1} x_1 +2\sum_{j\in N\setminus\{1\}} a_{ij} x_j(\boldsymbol{\mu}) = b_i & i\in I \\[6pt]
&\xi_i w_1  + 2 a_{i1} x_1 +2\sum_{j\in N\setminus\{1\}} a_{ij} x_j(\boldsymbol{\mu})\leq  b_i & i\in M\setminus I \\[6pt]
%\sum_{j\in N} x_j^2 \leq w &   \\ [6pt]
& \mu_i\geq 0 & i\in M.
\end{array}
$$
We notice that from the two equations associated to subset $I$, we can compute the solution of the corresponding linear system with the two unknowns $x_1$ and $w_1$, which is unique possibly after the application of the perturbation argument,  and we denote it by $(x_1(\boldsymbol{\mu}), w_1(\boldsymbol{\mu}))$, where
both $x_1(\boldsymbol{\mu})$ and $w_1(\boldsymbol{\mu})$ are affine functions of $\boldsymbol{\mu}$.
Then, we proved the following exactness condition.
\begin{prop}
\label{prop:exactref}
For $|H|=1$ the Shor relaxation is exact if for all $I\subset M$ with $|I|=2$, each of the following convex problems either has empty feasible region or has nonnegative optimal value:
\begin{equation}
\label{eq:auxprob}
\begin{array}{lll}
\min & \sum_{j\in N} x_j(\boldsymbol{\mu})^2 - w_1(\boldsymbol{\mu}) & \\ [6pt]
&d_{1}^*  +\sum_{i\in M} \xi_i \mu_i=0 & \\ [6pt]
&c_1 + \sum_{i\in M} a_{i1} \mu_i=0 &  \\[6pt]
&\xi_i w_1(\boldsymbol{\mu})  +2\sum_{j\in N} a_{ij} x_j(\boldsymbol{\mu})\leq  b_i & i\in M\setminus I \\[6pt]
%\sum_{j\in N} x_j^2 \leq w &   \\ [6pt]
& \mu_i\geq 0 & i\in M.
\end{array}
\end{equation}
\end{prop}
In what follows we provide an example where exactness cannot be established by the result stated in Section \ref{sec:h=1} but can be established by Proposition \ref{prop:exactref}.
\begin{exam}
Let us consider again problem (\ref{eq:param}) from Example \ref{ex:1}.
The convex relaxation (\ref{eq:newconvrel}) of that problem is:
$$
\begin{array}{ll}
\min & -w_1+\frac{1}{2} x_2^2 +x_2 \\ [6pt]
& w_1 +x_1-x_2 \leq 2 \\ [6pt]
& -x_1+x_2\leq \xi \\ [6pt]
& x_1^2+x_2^2\leq w_1. 
\end{array}
$$
%The feasible region is not empty. 
%Now, the set (\ref{eq:setspec}) in this case is:
%$$
%\begin{array}{l}
%-1+\mu_1=1 \\ [6pt]
%\mu_1-\mu_2=0 \\ [6pt]
% z_1+1 +x_1+1 \leq 2 \\ [6pt]
% -x_1-1\leq 1 \\ [6pt]
%x_1^2\leq z_1,
%\end{array}
%$$
%which contains, e.g., the point $\mu_1=\mu_2=1, x_1=-\frac{1}{2},z_1=\frac{1}{2}$.
%Thus, the exactness condition stated in Section \ref{sec:h=1} does not hold. 
%Note that exactness cannot be established by the conditions proposed in  \cite{BurerYe19}, \cite{Jeyakumar18} and \cite{Karzan20}. Indeed, these conditions do not depend on the right-hand sides of the constraints and, thus, they do not hold since we have already shown that they do not hold for Example \ref{ex:1}. For what concerns Proposition \ref{prop:exactref}, we first notice that we can only take $I=\{1,2\}$, so that in problem (\ref{eq:auxprob}) we have that $M\setminus I=\emptyset$,
%while $\mu_1=\mu_2=1$, $x_1(\mu_1,\mu_2)=-2$, $x_2(\mu_1,\mu_2)=-1$, and $w_1(\mu_1,\mu_2)=3$. Then, the optimal value of problem (\ref{eq:auxprob}) is equal to 2 and, thus, exactness holds.
%Note that, since $|M|=2$, here we could also have employed the exactness condition stated in Section \ref{sec:M=2}.
As already discussed, the exactness condition stated in Section \ref{sec:h=1} does not hold for all $\xi\geq -1$. 
Also recall that exactness cannot be established by the conditions proposed in  \cite{BurerYe19}, \cite{Jeyakumar18} and \cite{Karzan20} for all possible $\xi$ values, since  these conditions do not depend on the right-hand sides of the constraints.
For what concerns Proposition \ref{prop:exactref}, we first notice that we can only take $I=\{1,2\}$, so that in problem (\ref{eq:auxprob}) we have that $M\setminus I=\emptyset$,
while $\mu_1=\mu_2=1$, $x_1(\mu_1,\mu_2)=-1-\xi$, $x_2(\mu_1,\mu_2)=-1$, and $w_1(\mu_1,\mu_2)=2+\xi$. Then, the optimal value of problem (\ref{eq:auxprob}) is equal to $\xi^2+\xi$ and, thus, exactness holds for all
$\xi\leq -1$ and all $\xi\geq 0$. 
Note that, since $|M|=2$, here we could also have employed the exactness condition stated in Section \ref{sec:M=2}.
For $\xi\in (-1,0)$ the exactness condition does not hold but, actually, this happens since the bound provided by the convex relaxation in these cases is not tight. Indeed, the optimal value of the convex relaxation is equal to $-\frac{5}{2}-\xi$, attained at the given point $x_1=-1-\xi$, $x_2=-1$, $w_1=2+\xi$, 
while the optimal value of problem (\ref{eq:param}) can be seen to be equal to:
$$
-\frac{3}{2}-\frac{\xi}{4}-\frac{1}{2}\left(1+\frac{\xi}{2}\right)\sqrt{4+2\xi-\xi^2},
$$ 
attained at point 
$$
x_1^*=\frac{-\xi-\sqrt{4+2\xi-\xi^2}}{2},\ \ \ x_2^*=\frac{\xi-\sqrt{4+2\xi-\xi^2}}{2},
$$
where both constraints are active.
\end{exam}
We could go even further with a complete enumeration of
all subsets $I\subseteq M$ and assuming $\mu_i=0$ for all $i\in M\setminus I$. More precisely, 
let $2^M$ be the power set of $M$. Then, the relaxation is exact if for each $I\in 2^M$, $|I|\geq 2$, the following convex problems have nonnegative optimal values:
$$
\begin{array}{lll}
\min & \sum_{j\in N\setminus \{1\}}  x_j(\boldsymbol{\mu})^2 +x_1^2 - w_1 & \\ [6pt]
&d_{1}^*  +\sum_{i\in I} \xi_i \mu_i=0 & \\ [6pt]
&c_1 + \sum_{i\in I} a_{i1} \mu_i=0 &  \\[6pt]
%&(D_{jj}-d_{1}^*) x_j + c_j + \sum_{i\in I} a_{ij} \mu_i=0 & j\in N\setminus \{1\} \\[6pt]
&\xi_i w_1  + 2 a_{i1} x_1 +2\sum_{j\in N\setminus\{1\}} a_{ij} x_j(\boldsymbol{\mu}) = b_i & i\in I \\[6pt]
&\xi_i w_1  + 2 a_{i1} x_1 +2\sum_{j\in N\setminus\{1\}} a_{ij} x_j(\boldsymbol{\mu})\leq  b_i & i\in M\setminus I \\[6pt]
%\sum_{j\in N} x_j^2 \leq w &   \\ [6pt]
& \mu_i\geq 0 & i\in I.
\end{array}
$$
The obvious drawback of this condition is that it becomes unpractical when $|M|$ is large.

\section{Conclusion}
In this work we have shown that exactness results for the Shor relaxation of diagonal QCQPs 
can be derived by first proving the equivalence of this relaxation with two convex quadratic relaxations, and then by analyzing the KKT systems of these convex relaxations.
All this allows to re-derive previous exactness results in the literature and, in some cases, to strengthen them into primal-dual exactness conditions, i.e., conditions based both on the original (primal) variables of the convex relaxations and on the dual variables (Lagrange multipliers). As  a possible topic for future research we mention the possibility of extending the exactness results to non-diagonal QCQPs. In fact, as already mentioned, the result in \cite{Karzan20} already covers  some non-diagonal cases. It could be interesting to see whether the derivation discussed in this paper could be extended, e.g., to block diagonal QCQPs, by first proving the equivalence between the Shor relaxation and a convex program where  a distinct semidefinite condition is imposed for each distinct block, and then deriving optimality conditions for the convex program.

%\bibliographystyle{spmpsci} 
%\bibliography{ExactBiblio}  

\end{document}